\documentclass{amsart}

\usepackage{amssymb,amsmath,color,ytableau,graphicx,float,comment,longtable}
\usepackage{amsthm}
\usepackage{url}
\usepackage{tikz}
\usepackage{hyperref}

\definecolor{red}{rgb}{1,0,0}
\definecolor{blue}{rgb}{.2,.2,.8}

\def\ped{\textrm{ped}}
\def\pend{\textrm{pend}}
\def\pond{\textrm{pond}}
\def\pd{\textrm{pd}}
\def\pe{\textrm{pe}}
\def\PD{\mathcal{PD}}

\newtheorem{theorem}{Theorem}[section]
\newtheorem{corollary}[theorem]{Corollary}
\newtheorem{proposition}[theorem]{Proposition}

\theoremstyle{definition}

\newtheorem{example}{Example}
\newtheorem{remark}{Remark}

\begin{document}

\title{Generalizations of POD and  PED partitions  }
\author{Cristina Ballantine}\address{Department of Mathematics and Computer Science\\ College of the Holy Cross \\ Worcester, MA 01610, USA \\} 
\email{cballant@holycross.edu} 
\author{Amanda Welch} \address{Department of Mathematics and Computer Science\\ Eastern Illinois University \\ Charleston, IL 61920, USA \\} \email{arwelch@eiu.edu}

\maketitle

\begin{abstract}Partitions with even (respectively odd) parts distinct and all other parts unrestricted are often referred to as PED (respectively POD) partitions. In this article, we generalize these notions and study sets of partitions in which parts with fixed residue(s) modulo $r$ are distinct while all other parts are unrestricted. We also study partitions in which parts divisible by $r$ (respectively congruent to $r$ modulo $2r$)  must occur with multiplicity greater than one.

\
\end{abstract}

\allowdisplaybreaks

\section{Introduction}

A partition of a non-negative integer $n$ is a non-increasing sequence of  positive integers that sum  to $n$. We write a partition of $n$  as $\lambda = (\lambda_1, \lambda_2, \cdots, \lambda_{\ell})$ with $\lambda_1\geq \lambda_2\geq \ldots \geq \lambda_\ell$ and $\lambda_1 + \lambda_2 + \dotsm + \lambda_{\ell} = n$. We refer to the terms $\lambda_i$ as the parts of $\lambda$. As usual, we denote the number of partitions of $n$ by $p(n)$. 

Partition identities, statements about the number of partitions of $n$ with different sets of restrictions on the parts, are abundant in the theory of partitions. For an excellent introduction to the topic we refer the reader to \cite{A98}. 

Historically,  the first partition identity is perhaps Euler's theorem that states that, for $n \geq 0$, the number of partitions of $n$ with all parts odd equals the number of partitions of $n$ with all parts distinct. This result relates partitions with a parity condition to partitions with a multiplicity of parts condition. Many related results have been discovered and their study continues today. Of particular recent (and not so recent) interest have been partitions in which parts of fixed parity are required to be distinct, i.e., occur in the partition with multiplicity one. In the literature, partitions in which even (respectively odd) parts are distinct and all other parts are unrestricted are referred to as PED (respectively POD) partitions. 
 Much work has been done on these types of partitions. Identities as well as arithmetic properties are studied, for example, in \cite{A09}, \cite{A10}, \cite{BM21}, \cite{BW23}, \cite{M17}. Moreover, PED partitions were used by Andrews  \cite{A72} in the combinatorial proof of one of Gauss' theorems expressing Jacobi theta functions as infinite products.   

In this paper, we introduce generalizations of PED and POD partitions. We fix a modulus $r\geq 2$ and a residue $0\leq t < r$ and require parts congruent to $t$ modulo $r$ to be distinct while all other parts are unrestricted. When $r=2$ and $t=0$ (respectively $t=1$), this definition describes the PED (respectively POD) partitions. 

Recall that Glaisher's generalization of Euler's theorem states that the number of partitions of $n$ with no part congruent to $0$ modulo $r$ equals the number of partitions of $n$ where all parts must occur less than $r$ times. This leads us naturally to consider another  generalization of PED partitions by requiring that parts congruent to $0$ modulo $r$ occur less than $r$ times while all other parts are unrestricted.  We also consider other natural generalizations of POD partitions. 

It is worth noting that there has been other recent work  \cite{{A22}, {K}} in a related yet different direction: the study or partitions that are both $k$-regular, i.e., there are no parts divisible by $k$, and $s$-distinct, i.e., parts differ by at least $s$. The generalizations of PED and POD partitions introduced here are new and we hope that, in addition to our results, many new properties will be discovered.

The paper is structured as follows. In Section \ref{s_2}, we provide the necessary background material, including a list of frequently used notation. In Section \ref{s_3}, we study the properties of partitions in which parts congruent to $t$ modulo $r$ are distinct and all other parts are unrestricted. We extend known results for PED and POD partitions to this new set of partitions. In the particular case $t=0$, we obtain several recurrence relations. In Section \ref{gen-ped-gl}, we  introduce a generalization of PED partitions inspired by Glaisher's identity mentioned above.   In Section \ref{s_4}, we introduce two additional generalizations of POD partitions: partitions in which parts not congruent to $0$ modulo $r$ are distinct and parts divisible by $r$ have unrestricted multiplicity, and partitions in which parts congruent to $\pm t$  modulo $r$ are distinct and all other parts are unrestricted. We prove several theorems about these classes of partitions. 
One such result (Theorem \ref{thirdpod} in Section \ref{s_4}) relates the number $pd_{\pm t,r}(n)$ of partitions with parts congruent to $\pm t$ modulo $r$ distinct and all other parts unrestricted to the number $p_{\overline{\pm 2t},2r}(n)$ of partitions with no parts congruent to $\pm 2t$ modulo $2r$.
\begin{theorem}
    If $n \geq 0$ and $0 < t < r / 2$, $pd_{\pm t,r}(n) = p_{\overline{\pm 2t},2r}(n).$
\end{theorem}
If $r = 6$ and $t = 1$, this becomes $pd_{\pm 1,6}(n) = p_{\overline{\pm 2},12}(n).$ This sequence  can be found in \cite[A265254]{OEIS} where it is also described as the number of partitions of $n$ with even parts not distinct. This led us to investigate, in Section \ref{s_5},  PEND  (respectively POND) partitions, i.e, partitions with even (respectively odd) parts  not distinct, and their generalizations to partitions with parts congruent to $0$ modulo $r$ (respectively $r$ modulo $2r$) not distinct.
In Section \ref{s_6}, we give Beck-type identities for some of the identities introduced in this article. These are results giving a combinatorial interpretation for the excess in the total number of parts in all partitions of $n$ in one set of partitions versus another when the sets of partitions of $n$ are equinumerous.  Lastly,  in Section \ref{s_7}, we provide some concluding remarks and offer further avenues of inquiry for continuing the study of generalizations of  PED and POD partitions.

Whenever possible, we give both analytic and combinatorial proofs of our results.

\section{Background and Notation}\label{s_2}

In this section, we introduce notation and we briefly describe some necessary background on partitions. For a more extensive introduction to the theory of partitions, we refer readers to \cite{{A98},{AE}}. 

Given a partition $\lambda = (\lambda_1, \lambda_2, \cdots, \lambda_{\ell})$,  the size of $\lambda$, denoted $|\lambda|$, is the sum of all parts, i.e., the number being partitioned. We also write $\lambda\vdash n$ to mean that $\lambda$ is a partition of $n$. Occasionally, we work with pairs (or tuples) of partitions and abuse notation to write $(\lambda, \mu)\vdash n$ if $|\lambda|+|\mu|=n$ (and similarly for tuples). The number of parts in $\lambda$ is denoted by $\ell(\lambda)$ and we refer to it as the length of $\lambda$. When writing partitions, we sometimes use exponent notation to represent the multiplicity of parts, i.e., $(3^2, 2)$ denotes the partition $(3, 3, 2)$. We refer to a partition with all parts odd as an odd partition and to a partition with all parts distinct as a distinct partition. 

We use  several operations on partitions which we describe below. 

We often identify partitions with their multiset of parts. If $a$ is a positive integer, we write $a\in \lambda$ to mean that $a$ is a part of $\lambda$. 
Given two partitions $\lambda, \mu,$ we define $\lambda \cup \mu$ to be the multiset union of the parts of $\lambda$ and $\mu$, i.e., the  partition that consists of the parts of $\lambda$ and $\mu$  (with multiplicity) placed in non-increasing order. If $\mu$ is  a submultiset  of $\lambda$, we define $\lambda \setminus \mu$ to be the partition obtained by removing the parts of $\mu$ (with multiplicity) from the parts of $\lambda$.  If $m$ is a positive integer, $m \lambda$ denotes the partition whose parts are the the parts of $\lambda$ multiplied by $m$. 
If $\lambda$ is a partition with all parts divisible by $m$, we write $\frac{1}{m} \lambda$  for the partition whose parts are  the parts of $\lambda$ divided by $m$.  Often we write $\lambda = (\lambda^{r \mid}, \lambda^{r \nmid})$ where $\lambda^{r \mid}$ consists of the parts of $\lambda$ divisible by $r$  and $\lambda^{r \nmid}$ consists of the parts of $\lambda$ not divisible by $r$. As noted above, a partition $\lambda$ is said to be $r$-regular if it contains no parts divisible by $r$. 

Necessarily, this article contains a substantial amount of new notation. To help the reader, we first explain our rules for notation used throughout the paper. In the notation for a set of partitions, we use $\mathcal P$ to stand for ``partitions,"   letters following $\mathcal{P}$ place a restriction on certain parts of the partition, and the subscript describes which parts are restricted. The subscript consists of two parts: the residue(s) and the modulus. If the first term in the subscript is overlined, then  the restriction pertains to all parts not congruent to the overlined residue(s). For example, $\mathcal{PD}_{0, r}(n)$ is the set of partitions of $n$ where parts congruent to $0$ modulo $r$ must be distinct and $\mathcal{PD}_{\overline{0}, r}(n)$ is the set of partitions of $n$ where parts not congruent to $0$ modulo $r$ must be distinct.  We use $\mathcal{PD}$ for partitions in which certain parts must be distinct, $\mathcal{PND}$ for partitions in which certain parts cannot be distinct (i.e., must repeat), $\mathcal{PE}$ for partitions in which only certain even parts may be used, and $\mathcal{PEM}$ for partitions in which  certain parts  must have even multiplicity. Further, if there are additional restrictions on the entire set of partitions, then a superscript will be used. For example, we use $\mathcal{PD}^r_{0, 2r}(n)$ to denote the set of partitions of $n$ with all parts divisible by $r$ and parts congruent to $0$ mod $2r$ distinct. If the only letter in the description of the set is  $\mathcal{P}$ followed by a subscript, the set consists of all partitions whose parts are described by the subscript.  In addition, we use $\mathcal{Q}$ for the set of partitions with all parts distinct. We omit the size of the partitions from the notation if we mean the set of all partitions with a given description. For example $$\mathcal{PD}_{0,r}=\bigcup_{n\geq 0}\mathcal{PD}_{0, r}(n).$$
\begin{example}
    Among the partitions in $\mathcal{PD}_{0, 3}(9)$ are  $(9), (6,3), (6,2,1), (5,2,1,1)$, and $(3, 1^6)$. The partitions  $(3^3), (3^2, 2, 1),$ and $(3^2, 1^3)$ are not in $\mathcal{PD}_{0, 3}(9)$ as they contain repeated parts congruent to $0$ mod $3$. 
    Among the partitions in  $\mathcal{PD}_{\overline{0}, 3}(9)$ are $(3^3)$ and $(3^2, 2, 1)$. The partitions  $(3^2, 1^3)$ and $(2^3, 1)$ are not in $\mathcal{PD}_{\overline{0}, 3}(9)$ as they contain repeated parts that are not congruent to $0$ mod $3$. 
   \end{example}

\noindent \textit{Notation List:} Let $n\geq 0$, $r\geq 2$ and $0\leq t < r$  be integers. The following lists our frequently used notation for sets of partitions.

\begin{longtable}{p{.8in} p{3.6in}}
$\mathcal{P}_{t, r}(n)$ & partitions of $n$ in which all parts are congruent to $t\mod r$.  \\

$\mathcal{P}_{\overline{t}, r}(n)$ & partitions of $n$ in which no parts are congruent to $t\mod r$. \\

$\mathcal{Q}_{\overline{0}, 2r}(n)$ & distinct partitions of $n$ with no parts congruent to $0 \mod{2r}$.\\

$\mathcal{PD}_{t,r}(n)$ & partitions of $n$ in which all parts congruent to $t\mod r$ are distinct and all other parts are unrestricted.  \\

$\mathcal{PD}_{\bar{t},r}(n)$ & partitions of $n$ in which all parts not congruent to $t\mod r$ are distinct and parts congruent to $t
\mod r$ are unrestricted. \\ 

$\mathcal{B}_{r}(n)$ & $r$-regular partitions of $n$, i.e., partitions of $n$  with no parts congruent to $0 \mod{r}$.  \\

$\mathcal{PE}_{0, 2r}(n)$ &  partitions of $n$ where all even parts are congruent to $0$ mod $2r$ and all other parts are unrestricted.\\

$\mathcal{QE}_{0, 2r}(n)$ &  distinct partitions of $n$ where all even parts are congruent to $0$ mod $2r$. \\

$\mathcal{PD}^r_{0, 2r}(n)$ & partitions of $n$ with all parts divisible by $r$ and parts congruent to $0$ mod $2r$ distinct.\\

$\mathcal{PND}_{0, r}(n)$ & partitions of $n$ where parts congruent to $0$ mod $r$ are not distinct and all other parts are unrestricted.\\

$\mathcal{PND}_{\overline{0}, r}(n)$ & partitions of $n$ where parts not congruent to $0$ mod $r$ are not distinct and all other parts are unrestricted.\\

$\mathcal{PEM}_{\pm r, 3r}(n)$ & partitions of $n$ where parts congruent to $\pm r \mod{3r}$ have even multiplicity and all other parts are unrestricted.\\

$\mathcal{PRM}_{0,r}(n)$ & partitions of $n$ where parts congruent to $0$ modulo $r$ have multiplicity less than $r$ and all other parts are unrestricted.\\

\end{longtable}

Note that $\mathcal B_{r}(n)=\mathcal P_{\overline{0}, r}(n)$. However, $r$-regular partitions have been studied extensively and the notation $\mathcal B_{r}(n)$ is customary. 

We denote the number of partitions in a given set by corresponding lowercase letters. For example, $\textrm{pd}_{0,r}(n)$ denotes the number of partitions of $n$ in which all parts congruent to $0\mod r$ are distinct and all other parts are unrestricted. We make an exception for distinct partitions, using $Q$ instead of $q$ to avoid confusion with the variable in $q$-series. We denote by $\textrm{ped}(n)$ (respectively $\textrm{pod}(n)$) the number of PED (respectively POD) partitions of $n$. 

The index $e-o$ appended to a subscript of a partition number indicates that each partition $\lambda$   is counted with weight $(-1)^{\ell(\lambda)}$. For example, $\pd_{0,r,e-o}(n)$ is the number of partitions in $\PD_{0,r}(n)$ with an even number of parts minus the number of partitions in $\PD_{0,r}(n)$ with an odd number of parts, i.e.,  $$\pd_{0,r,e-o}(n)=|\{\lambda\in \PD_{0,r}(n) \mid \ell(\lambda) \text{ even}\}|- |\{\lambda\in \PD_{0,r}(n) \mid \ell(\lambda) \text{ odd}\}|.$$

The set of positive integers is denoted by $\mathbb N$ and the set of non-negative  integers is denoted by  $\mathbb N_0$. 

We use the Pochhammer symbol notation
\begin{align*}
	& (a;q)_n = \begin{cases}
	1 & \text{for $n=0$,}\\
	(1-a)(1-aq)\cdots(1-aq^{n-1}) &\text{for $n>0$;}
	\end{cases}\\
	& (a;q)_\infty = \lim_{n\to\infty} (a;q)_n.
	\end{align*}
	We assume $|q|<1$, so all series converge absolutely. 
	
	In several of our combinatorial proofs we make use of Glaisher's bijection $\varphi_G$ \cite{Gl}. If we denote by $\mathcal D_r(n)$ the set of partitions of $n$ with parts occurring less than $r$ times, 
	$$\varphi_G:\mathcal B_r(n)\to \mathcal D_r(n) $$ takes a partition into parts not divisible by $r$ and merges $r$ equal parts into a single part repeatedly until the transformed partition has no parts repeated at least $r$ times. 
 
 If $r=2$, Glaisher's bijection gives a combinatorial proof of Euler's identity. 

\section{Generalizations of ped and pod partitions}\label{s_3}

Let $n\geq 0$, $r\geq 2$ and $0\leq t < r$ be integers and denote by $\mathcal{PD}_{t,r}(n)$ the set of partitions of $n$ in which parts congruent to $t$ modulo $r$ are distinct and all other parts are unrestricted. Set $\pd_{t,r}(n):=|\mathcal{PD}_{t,r}(n)|$. As noted in the introduction, if $r=2$ and $t=0$ (respectively $t=1$), $\mathcal{PD}_{t,r}(n)$ is the set of PED (respectively POD) partitions of $n$.

Let  $\mathcal{P}_{\overline{2t},2r}(n)$ be the set of partitions of $n$ with no parts congruent to $2t$ modulo $2r$ and set $p_{\overline{2t},2r}(n):=|\mathcal{P}_{\overline{2t},2r}(n)|$.  
\begin{theorem}\label{1st_gen} 
If $n\geq 0$, $r\geq2$,  and $0 \leq t < r$, we have $pd_{t,r}(n)=p_{\overline{2t},2r}(n)$. 
\end{theorem}
\begin{proof}[Analytic proof]
\begin{align*}\sum_{n=0}^\infty pd_{t,r}(n)q^n & = \frac{(q^t;q^r)_\infty}{(q;q)_\infty}\cdot(-q^t;q^r)_\infty\\
 & = \frac{(q^{2t};q^{2r})_\infty}{(q;q)_\infty}\\ & =\sum_{n=0}^\infty p_{\overline{2t},2r}(n)q^n.
\end{align*}\end{proof}
\begin{proof}[Combinatorial proof]   We create a bijection  from $\lambda \in \PD_{t,r}(n)$ to $\mathcal{P}_{\overline{2t},2r}(n)$.  

\underline{Case 1:} $t\neq 0$. Let $\lambda \in \PD_{t,r}(n)$. Replace each part $\lambda_i$ of $\lambda$ that is congruent to $2t$ modulo $2r$ by two  parts  equal to $\lambda_i/2$ to obtain a partition $\mu\in \mathcal{P}_{\overline{2t},2r}(n)$. Conversely, if  $\mu\in \mathcal{P}_{\overline{2t},2r}(n)$, write $\mu=\alpha \cup \beta$ where $\alpha$ is a distinct partition into parts congruent to $t$ modulo $r$ and $\beta$ is a partition in which parts congruent to $t$ modulo $r$ appear with even multiplicity. Replace each pair of parts congruent to $t$ modulo $r$  in $\beta$ by a part equal to their sum to obtain a partition $\gamma$ with no parts congruent to $t$ modulo $r$. Then $\alpha\cup \gamma \in \PD_{t,r}(n)$.

\underline{Case 2:} $t= 0$. Let $\lambda \in \PD_{0,r}(n)$. As explained in Section \ref{s_2}, we write $\lambda=\lambda^{r \mid }\cup \lambda^{r\nmid }$, where $\lambda^{r\mid }$ (respectively $\lambda^{r\nmid }$) is the partition consisting of the parts of $\lambda$ that are divisible (respectively not divisible) by $r$. We write $\lambda^{r\mid}=r\mu$, where $\mu$ is a distinct partition.  The inverse of Glaisher's bijection, $\varphi_G^{-1}$, maps $\mu$   to an odd partition $\varphi_G^{-1}(\mu)$. Then, the parts of $r\varphi_G^{-1}(\mu)$ are divisible by $r$ but not by $2r$ and $r\varphi_G^{-1}(\mu)\cup \lambda^{r\nmid}$ is a partition in $\mathcal{P}_{\overline 0,2r}(n)$.  For the inverse, start with a partition $\eta$ of $n$ with parts not divisible by $2r$ and write $\eta=(\eta^{r\mid}, \eta^{r\nmid})$. Then $\eta^{r\mid}=r\nu$ where $\nu$ is an odd partition. Since $\varphi_G(\nu)$ is a partition with distinct parts, it follows that $r\varphi_G(\nu)\cup \eta^{r\nmid}$ is a partition in $\PD_{0,r}(n)$.
\end{proof}

Recall that $\mathcal B_{2r}(n)$ is the set of $2r$-regular partitions of $n$ and $b_{2r}(n)=|\mathcal B_{2r}(n)|$. 
\begin{corollary}\label{Th_prd} If $n\geq 0$ and $r\geq2$, we have $pd_{0,r}(n)=b_{2r}(n)$. 
\end{corollary}
Corollary \ref{Th_prd} allows us to use  existing results  on the parity of the number of $\ell$-regular  partitions  to infer parity results for $pd_{0,r}(n)$. For example, it is well known that $b_4(n)$ is odd if and only if $n$ is a triangular number (see for example \cite[Theorem 3.1]{BM17}). Thus, as stated in \cite[Corollary 1.3]{M17}, $\ped(n)$ is odd if and only if $n$ is a triangular number. From the parity criterion for  $b_8(n)$ obtained  in \cite{CM} using modular forms (and also proved in \cite{KZ} via elementary methods), we have that  $pd_{0,4}(n)$ is odd if and only if $24n+7=p^{4a+1}M^2,$ for some prime $p\nmid M$ and some $a\geq 0$.

We denote by $\mathcal{QE}_{2t, 2r}(n)$  the set of distinct partitions of $n$ with even parts congruent to $2t$ modulo $2r$ and set $QE_{2t, 2r}(n):=|\mathcal{QE}_{2t, 2r}(n)|$. We have the following result. 
\begin{theorem}\label{prd_eo} For $n\geq 0$, $r\geq 2$, and $0\leq t<r$, we have $pd_{t,r, e-o}(n)=QE_{2t, 2r, e-o}(n)$. 
    \end{theorem}
    \begin{proof}[Analytic proof]\begin{align*}\sum_{n=0}^\infty pd_{t,r, e-o}(n)q^n & = \frac{(-q^t;q^r)_\infty}{(-q;q)_\infty}\cdot (q^t;q^r)_\infty\\
 & = \frac{(q^{2t};q^{2r})_\infty}{(-q;q)_\infty}\\ & = (q^{2t};q^{2r})_\infty(q;q^2)_\infty\\ & =\sum_{n=0}^\infty QE_{2t,2r, e-o}(n)q^n.
\end{align*} In the penultimate equality we used Euler's theorem \cite[Identity (1.2.5)]{A98} \begin{equation}\label{Euler}\frac{1}{(q;q^2)_\infty}=(-q;q)_\infty.\end{equation}
    \end{proof}
    \begin{proof}[Combinatorial proof] We can also prove the theorem adapting the combinatorial argument of \cite{G75}. Let $\mathcal S(n)$ be the subset of partitions $\pi\in \mathcal{PD}_{t,r}(n)$   such that there is at least one repeated part or at least one even part not congruent to $2t$ modulo $2r$. Then $$\mathcal{PD}_{t,r}(n)\setminus \mathcal S(n)=\mathcal{QE}_{2t, 2r}(n).$$
    We define a sign reversing involution $\psi$ on $\mathcal S(n)$ as follows. If $\pi\in \mathcal S(n)$, we denote by $a=a(\pi)$ the largest repeated part of $\pi$  and by $b=b(\pi)$ the largest even part of $\pi$ that is not congruent to $2t$ modulo $2r$. If $a$ or $b$ do not exist, we set  $a=0$ or $b=0$, respectively. 

    If $2a>b$, define $\psi(\pi)=\pi\setminus(a,a)\cup(2a)$. Since $a\not\equiv t\mod r$, we have $2a\not\equiv 2t\mod{2r}$ and $\psi(\pi)\in \mathcal S(n)$.

    If $2a\leq b$, define $\psi(\pi)=\pi\setminus(b)\cup(b/2,b/2)$. Since $b\not \equiv 2t\mod{2r}$, we have $b/2\not\equiv t\mod r$ and $\psi(\pi)\in \mathcal S(n)$.

    Clearly, $\psi$ is an involution of $\mathcal S(n)$ that reverses the parity of $\ell(\pi)$. This completes the combinatorial proof. 
    \end{proof}
   
For the remainder of this section, we restrict our attention to the case $t=0$. 
\begin{corollary} \label{Th_prd_eo} If $n\geq 0$ and $r\geq 2$, $pd_{0,r, e-o}(n)=QE_{0, 2r, e-o}(n)$. 
\end{corollary}
Since $\pd_{0, r}(n)=b_{2r}(n)$, it is natural to consider the analogous result for $2r$-regular partitions. We note that the  case $r=2$ of the result below appears in \cite[Theorem 1.1(b)]{BM21}.

\begin{proposition}\label{Th_b2r_eo} For $n\geq 0$, $r\geq 2$ we have 
$b_{2r,e-o}(n) = (-1)^nQE_{0,2r}(n)$.
\end{proposition}
\begin{proof}[Analytic proof] As in the proof of Theorem \ref{prd_eo}, using Euler's identity \eqref{Euler}, we obtain 
    \begin{align*}
        \sum_{n=0}^\infty b_{2r,e-o}(n)q^n& =\frac{(-q^{2r};q^{2r})_\infty}{(-q;q)_\infty} = (-q^{2r};q^{2r})_\infty(q;q^2)_\infty.
    \end{align*}
    Since the number of odd parts in a partition has the same parity as the size of the partition, the proof is complete. 
\end{proof}

\begin{proof}[Combinatorial proof] \underline{Case 1:} $r$ is  odd. We proceed as in the proof of Theorem \ref{prd_eo}. We take $\mathcal S(n)$ to be the subset of partitions $\pi\in\mathcal B_{2r}(n)$ such that $\pi$ has at least one part not congruent to $0$ modulo $r$ repeated or at least one even part.  We denote by $a=a(\pi)$ the largest repeated part of $\pi$ that is not congruent to $0$ modulo $r$  and by $b=b(\pi)$ the largest even part of $\pi$. If $a$ or $b$ do not exist, we set  $a=0$ or $b=0$, respectively.  We define a sign reversing involution $\psi$ in $\mathcal S(n)$ as in the proof of Theorem \ref{prd_eo}. This shows that $b_{2r,e-o}(n)$ equals the number of partitions $\pi$ of $n$ into  odd parts such that only parts congruent to $0$ modulo $r$ may repeat and $\ell(\pi)$ is even minus the number of partitions $\pi$ of $n$ into  odd parts such that only parts congruent to $0$ modulo $r$ may repeat and $\ell(\pi)$ is odd. Since for an odd partition $\pi$, $\ell(\pi)\equiv n\mod 2$, it follows that $b_{2r,e-o}(n)$ equals    $(-1)^n$ times  the number of partitions of $n$ into odd parts such that only parts congruent to $0$ modulo $r$ may repeat. Applying Glaisher's bijection to each partition $\pi$ of $n$ into odd parts such that only parts congruent to $0$ modulo $r$ may repeat completes the combinatorial proof.

\underline{Case 2:} $r$ is  even. For $r=2$ a combinatorial proof is given in \cite[Theorem 1.1(b)]{BM21}. The case $r=2$ of the proof below is different from the proof in \cite{BM21}.

Let $\mathcal S(n)$ be the subset of partitions   $\pi\in \mathcal B_{2r}(n)$ such that  $\pi$ has at least one part that is not congruent to $0$ modulo $r$ repeated or at least one even part with odd multiplicity. We denote by $a=a(\pi)$ the largest repeated part of $\pi$ that is not congruent to $0$ modulo $r$ and by $b=b(\pi)$ the largest even part with odd multiplicity. If $a$ or $b$ do not exist, we set  $a=0$ or $b=0$, respectively.  We define a sign reversing involution $\psi$ on $\mathcal S(n)$ as in the proof of Theorem \ref{prd_eo}. The partitions in $\mathcal B_{2r}(n)\setminus \mathcal S(n)$ are such that only parts congruent to $0$ modulo $r$ may be repeated and even parts appear with even multiplicity.  Therefore  even parts must be congruent to $0$ modulo $r$ (but not congruent to $0$ modulo $2r$).
Let $\widetilde{\mathcal{QE}}_{r,2r}(n)$ be the set of partitions of $n$ such that odd parts are distinct and  even parts are congruent to $r$ modulo $2r$ and have even multiplicity. Then $\mathcal B_{2r}(n)\setminus \mathcal S(n)= \widetilde{\mathcal{QE}}_{r,2r}(n)$ and the argument above shows that $b_{2r, e-o}(n)=\widetilde{QE}_{r,2r,e-o}(n)$.

Next, we define a bijection  between $\mathcal{QE}_{0,2r}(n)$ and $\widetilde{\mathcal{QE}}_{r,2r}(n)$. Start with  $\lambda=(\lambda^{2\mid}, \lambda^{2\nmid})\in \mathcal{QE}_{0,2r}(n)$. For each part $2rt$ in  $\lambda^{2\mid}$, we  write $2rt=2^irc$ with $c$ odd and $i\geq 1$ and replace part $2^irc$ by $2^i$ parts equal to $rc$. We denote the obtained partition by $\mu$. Then $rc\equiv r\mod 2r$ occurs with even multiplicity in $\mu$ and hence $\mu\in \widetilde{\mathcal{QE}}_{r,2r}(n)$. Conversely, if  $\mu\in \widetilde{\mathcal{QE}}_{r,2r}(n)$, then $\lambda:=\varphi_G(\mu)\in \mathcal{QE}_{0,2r}(n)$ is the partition obtained by iteratively merging parts congruent to $r$ modulo $2r$ in pairs until we obtain a distinct partition.

If $\lambda \in \widetilde{\mathcal{QE}}_{r,2r}(n)$, since even parts occur with even multiplicity, it follows that $\ell(\lambda)\equiv n\mod 2$. 
Hence $$b_{2r, e-o}(n)=\widetilde{QE}_{r,2r, e-o}(n)= (-1)^n \widetilde{QE}_{r,2r}(n)=(-1)^nQE_{0,2r}(n).$$
\end{proof}

\begin{corollary} For $n\geq 0$ and $r\geq 2$, $b_{2r}(n)$, $pd_{0,r}(n)$ and $QE_{0,2r}(n)$ have the same parity.  
In particular, 
$QE_{0,4}(n)$ is odd if and only if $n$ is a triangular number and $QE_{0,8}(n)$ is odd if and only if $24n+7=p^{4a+1}M^2,$ for some prime $p\nmid M$ and some $a\geq 0$.
\end{corollary}

Next, we consider recurrence relations involving $\pd_{0,r}(n)$. 
In \cite[Theorem 7]{dSD} the authors express the number of $\ell$-regular partitions in terms of $p(n-\ell G_k)$, where $G_k$ are generalized pentagonal numbers. Then,  Corollary \ref{Th_prd} leads to the following relation. 
\begin{theorem} Let $r\geq 2$. For any  $n\geq 0$ we have $$ pd_{0,r}(n)=p(n)+\sum_{j=1}^\infty (-1)^j\left( p(n-2rj(3j-1)/2)+ p(n-2rj(3j+1)/2)\right).$$
    \end{theorem}

    If we replace $4$ by $2r$ in the proof of \cite[Theorem 3.2]{FGK},  we obtain the following recurrence relation. 

    \begin{theorem} \label{pent} Let $r\geq 2$. For any  $n\geq 0$ we have
    $$\sum_{j=-\infty}^\infty (-1)^{j} pd_{0,r}(n-j(3j+1)/2)=\begin{cases} (-1)^k & \text{ if } n=rk(3k+1), k \in \mathbb Z\\ 0&  \text{ otherwise. }  \end{cases}$$
        
    \end{theorem}

In \cite{BW23} we gave two combinatorial proofs of Theorem \ref{pent} in the case $r=2$. The first proof makes use of Glaisher's identity. Replacing all partitions with parts repeated at most $3$ times by partitions with parts repeated at most $2r-1$ times, the proof can be easily generalized to a combinatorial proof of Theorem \ref{pent}.

The next theorem gives a recurrence relation involving the sequences $b_r(n)$ and $b_{2r}(n)$.

\begin{theorem} \label{Th_pd_b} 
For $n\geq 0$, we have $$\sum_{j=-\infty}^\infty (-1)^{j} b_{2r}(n-rj^2)= \sum_{j=-\infty}^\infty (-1)^j b_r(n-rj(3j+1)/2).$$
   \end{theorem}
       \begin{proof} From \cite[Identity (2.2.12)]{A98} and Euler's pentagonal number theorem \cite[Identity (5.13)]{AE}, both with   $q$ replaced by $q^r$, we obtain

\begin{align}\notag \sum_{n=0}^\infty \sum_{j=-\infty}^\infty (-1)^{j} b_{2r}(n-rj^2)q^n& =\left(\sum_{j=-\infty}^\infty(-1)^jq^{rj^2}\right)\left( \sum_{n=0}^\infty b_{2r}(n)q^n\right) \\ \notag & =
\frac{(q^r;q^r)_\infty}{(-q^r;q^r)_\infty}\cdot \frac{(q^{2r};q^{2r})_\infty}{(q;q)_\infty}\\ & \notag = (q^{r};q^{r})_\infty\frac{(q^{r};q^{r})_\infty}{(q;q)_\infty}\\ \notag & = \left(\sum_{j=-\infty}^\infty (-1)^j q^{rj(3j+1)/2}\right)\left(\sum_{n=0}^\infty b_r(n)q^n\right) \\ \notag  & =\sum_{n=0}^\infty\sum_{j=-\infty}^\infty (-1)^j b_r(n-rj(3j+1)/2)q^n.\end{align} 
 \end{proof}
 Rewriting the statement of Theorem \ref{Th_pd_b}, we obtain a recurrence relation involving the sequences $b_r(n)$ and $\pd_{0,r}(n).$ 
 \begin{corollary}\label{Th_r=2}
     For $n\geq 0$, we have \begin{equation}\label{b2} \sum_{j=-\infty}^\infty (-1)^{j} pd_{0,r}(n-rj^2)= \sum_{j=-\infty}^\infty (-1)^j b_r(n-rj(3j+1)/2).\end{equation}
 \end{corollary}
 \begin{remark}
     Since $b_2(n)$ is the number of partitions of $n$ into odd parts,  using Euler's pentagonal number theorem \cite[Identity (1.3.1)]{A98} and Gauss' theta function  identity  \cite[Identity (2.2.13)]{A98},  the right-hand side of \eqref{b2} becomes $$\sum_{j=-\infty}^\infty (-1)^j b_2(n-2j(3j+1)/2)=(q^2;q^2)_\infty \frac{1}{(q;q^2)_\infty}=\sum_{n=0}^\infty q^{n(n+1)/2}. $$
     
Then, Corollary  \ref{Th_r=2} for $r=2$  reduces to  \cite[Theorem 1.2]{M17}.
 \end{remark}

We conclude this section with another recurrence relation for $\pd_{0,r}(n)$ which is valid for odd $n$. 

\begin{theorem} If $n>0$ is odd and $r\geq 2$, we have \begin{equation}\label{Th_pd_rec}\pd_{0,r}(n)= \sum_{j=1}^\infty  (-1)^{\lceil\frac{j}{2}\rceil+1} \pd_{0,r}(n-j(j+1)/2).\end{equation} \end{theorem}
\begin{proof} Using the Jacobi Triple Product identity \cite[Theorem 2.8]{A98} with $z=-q$ and $q$ replaced by $q^2$,  we have 
\begin{align} \notag\sum_{n=0}^\infty \Big(\sum_{j=0}^\infty (-1)^{\lceil\frac{j}{2}\rceil} \pd_{0,r}(n &  -j(j+1)/2)\Big) q^n \\ \notag &  = \left(\sum_{j=0}^\infty (-1)^{\lceil\frac{j}{2}\rceil}q^{j(j+1)/2}\right)\left(\sum_{n=0}^\infty\pd_{0,r}(n)q^n\right)\\ \notag & =(q^4;q^4)_\infty(q; q^2)_\infty \frac{(q^{2r}; q^{2r})_\infty }{(q; q)_\infty} \\ \label{signedtr} & = \frac{(q^{2r}; q^{2r})_\infty }{(q^2; q^4)_\infty}.\end{align} If $n$ is odd,  the coefficient of $q^n$ in \eqref{signedtr} is zero, and thus if $n$ is odd, $$\sum_{j=0}^\infty (-1)^{\lceil\frac{j}{2}\rceil} \pd_{0,r}(n-j(j+1)/2)=0$$
and \eqref{Th_pd_rec} holds. 
\end{proof}

\section{A Second Generalization of ped partitions} \label{gen-ped-gl}

In this section, we consider another natural generalization of PED partitions: partitions whose parts divisible by $r$ occur with restricted multiplicity. We denote by $\mathcal{PRM}_{0,r}(n)$ the set of partition of $n$ whose parts congruent to $0$ modulo $r$ have multiplicity less than $r,$ and set $\textrm{prm}_{0,r}(n):=|\mathcal{PRM}_{0,r}(n)|$. If $r = 2$, $\mathcal{PRM}_{0,r}(n)$ is the set of PED partitions of $n$. 

\begin{theorem} \label{second pd}
    Let $n \geq 0$, $r\geq 2$. Then,  $\textrm{prm}_{0,r}(n)=b_{r^2}(n)$. 
   \end{theorem}

\begin{proof}[Analytic Proof.] The generating function for the number of partitions of $n$ with parts congruent to $0$ modulo $r$ repeated at most $r-1$ times is $$\frac{(q^r; q^r)_\infty}{(q; q)_\infty} \prod_{j=1}^\infty (1 + q^{rj} + q^{2rj} + \dotsm +q^{(r-1)rj})=\frac{(q^r; q^r)_\infty}{(q; q)_\infty} \frac{(q^{r^2}; q^{r^2})_\infty}{(q^r; q^r)_\infty} = \frac{(q^{r^2}; q^{r^2})_\infty}{(q; q)_\infty},$$
which is the generating function for the number of $r^2$-regular partitions of  $n$. \end{proof}
\begin{proof}[Combinatorial proof.] Let $\lambda\in \mathcal{PRM}_{0,r}(n)$ and write $\lambda = ( \lambda^{r \mid}, \lambda^{r \nmid})$. Then the parts of $\lambda^{r \nmid}$ have unrestricted multiplicity but the parts of $\lambda^{r \mid}$ all have multiplicity less than $r$. 
Since $\lambda^{r\mid}=r\eta$ for some partition $\eta$ with parts occurring less than $r$ times, we may apply the inverse of Glaisher's bijection to $\eta$. Then, $r\varphi_G^{-1}(\eta)$ is a partition with  parts congruent to $0$ modulo $r$ but not congruent to $0$ modulo $r^2$. Hence,  $\mu = \lambda^{r \nmid} \cup r\varphi_G^{-1}(\eta)\in \mathcal B_{r^2}(n)$.

For the inverse, start with a partition $\mu\in \mathcal B_{r^2}(n)$ and write $\mu = (\mu^{r \mid}, \mu^{r\nmid})$. Since $\mu^{r\mid}=r\nu$ for some partition $\nu\in \mathcal B_r$, we may apply Glaisher's bijection to $\nu$.  Then, $r\varphi_G(\nu)$  is a partition with all parts congruent to $0$ modulo $r$ and multiplicity less than $r$. Hence, $\lambda= \mu^{r \nmid} \cup r\varphi_G(\nu) \in \mathcal{PRM}_{0,r}(n)$.
\end{proof}

Corollary \ref{Th_prd} and Theorem \ref{second pd} lead to the following result.

\begin{corollary} If $n\geq 0$ and $r\geq 2$, then $\pd_{0,2r^2}(n)=\textrm{prm}_{0,2r}(n)$. 
    \end{corollary}

    A complete characterization of the divisibility of $b_9(n)$ by $3$ has been  given recently in \cite{Abinash}. It follows that $\textrm{prm}_{0,3}(n)$ is divisible by $3$ if and only if at least one of the following conditions holds:
    \begin{itemize}
\item[1.] There is a prime $p \equiv 2\pmod 3$  such that $\textrm{ord}_p(3n+1)$  is odd.
 \item[2.]  There is a prime $p \equiv 1\pmod 3$ such that $\textrm{ord}_p(3n+1)\equiv 2 \pmod 3$.
\end{itemize} Here, $\textrm{ord}_p(n)$ denotes the highest power of $p$ that divides $n$.

\begin{remark} The generalization introduced in this section could be extended further. For $n\geq 0$, $r\geq 2$ and $0\leq t<r$, let $\mathcal{PRM}_{t,r}(n)$ be the set of partitions of $n$ whose parts congruent to $t$ modulo $r$ have multiplicity less than $r,$ and set $\textrm{prm}_{t,r}(n):=|\mathcal{PRM}_{t,r}(n)|$. One can easily adapt  both proofs of Theorem \ref{second pd} to show that $\textrm{prm}_{t,r}(n)$ equals the number of partitions of $n$ with no parts congruent to $tr$ modulo $r^2$.  When $r=2$ and $t=1$,  $\mathcal{PRM}_{t,r}(n)$ is the set of POD partitions of $n$.

\end{remark}

\section{Further generalizations of pod partitions}\label{s_4}

 In this section, we introduce two additional generalizations of POD partitions, beginning with partitions with parts not divisible by $r$ distinct. For $n\geq 0$ and $r\geq 2$, we denote by $\mathcal{PD}_{\overline 0,r}(n)$ the set of partitions of $n$ with parts not congruent to $0$ modulo $r$ distinct and parts divisible by $r$ having unrestricted multiplicity. We set $\pd_{\overline 0, r}(n):=|\mathcal{PD}_{\overline 0,r}(n)|$. When $r=2$, $\mathcal{PD}_{\overline 0,r}(n)$ is  the set of POD partitions of $n$.
We prove several identities involving $\pd_{\overline 0, r}(n),$ some of which extend existing results for $\textrm{pod}(n)$. 

Let $\mathcal{PE}_{0,2r}(n)$ be the set of partitions of $n$ in which all even parts are divisible by $2r$ and odd parts are unrestricted, and set $\pe_{0,2r}(n):=|\mathcal{PE}_{0,2r}(n)|$.

\begin{theorem}\label{pe}
    If $n \geq 0$ and $r\geq 2$, we have $\pd_{\overline 0,r}(n) = \pe_{0,2r}(n).$
\end{theorem}
\begin{proof}[Analytic Proof.] 
\begin{align*}\sum_{n=0}^\infty \pd_{\overline 0,r}(n)q^n & = \frac{1}{(q^r;q^r)_\infty}\cdot \frac{(-q;q)_\infty}{(-q^r;q^r)_\infty}\\
 & = \frac{1}{(q^{2r};q^{2r})_\infty(q;q^{2})_\infty}\\ & =\sum_{n=0}^\infty \pe_{0,2r}(n)q^n.
\end{align*}
\end{proof}

\begin{proof}[Combinatorial Proof.]
Let $\lambda\in \mathcal{PD}_{\overline 0,r}(n)$.  As before, write $\lambda=(\lambda^{r \mid }, \lambda^{r\nmid })$. We further write $\lambda^{r \mid }=\alpha\cup \beta$ where   $\alpha$ is a partition with all parts occurring with even multiplicity and $\beta$ is a distinct partition. Then $\beta\cup \lambda^{r\nmid }$ is a distinct partition and therefore  $\varphi_G(\beta\cup \lambda^{r\nmid })$ is an odd partition. Let $\psi(\alpha)$ be the partition obtained from $\alpha$ by merging pairs of equal parts into a single part. Then $\psi(\alpha)\cup \varphi_G(\beta\cup \lambda^{r\nmid })\in \mathcal{PE}_{0,2r}(n)$.  The transformation is clearly invertible.
\end{proof}

\begin{example}
    Let $\lambda = (12^3, 9^4, 7, 3^5, 2, 1) \in \mathcal{PD}_{\overline{0}, 3}(97)$. Then $\lambda^{3 \mid} = (12^3, 9^4, 3^5)$ and $\lambda^{3 \nmid} = (7, 2, 1)$. Set $\alpha = (12^2, 9^4, 3^4)$ and $\beta = (12, 3)$. Then $\alpha \cup \beta = \lambda^{3 \mid}$, $\beta \cup \lambda^{3 \nmid} = (12, 7, 3, 2, 1)$ is a distinct partition, and every part in $\alpha$ has even multiplicity.
 Using Glaisher's transformation, we map $\beta \cup \lambda^{3 \nmid}$ to $\varphi_G(\beta \cup \lambda^{3 \nmid}) = (7, 3^5, 1^3)$, an odd partition. Merging in pairs, we map $\alpha$ to $\psi(\alpha) = (24, 18^2, 6^2),$ which is a partition with even parts, all congruent to $0$ mod $6$. 
  Thus, $\psi(\alpha) \cup \varphi(\beta \cup \lambda^{3 \nmid}) = (24, 18^2, 7, 6^2, 3^5, 1^3) \in \mathcal{PE}_{0, 6}(97)$.
\end{example}

\begin{remark}
    If $r=3$, $\pd_{\overline 0, r}(n)$ appears in \cite[A096981]{OEIS} with both interpretations: the number of partitions in $\mathcal{PD}_{\overline 0,r}(n)$ and also the number of partitions in $\mathcal{PE}_{0,2r}(n).$
For larger $r$, the sequence does not appear in \cite{OEIS}.
\end{remark}

\begin{remark} An overpartition of $n$ is a partition in which the first occurrence of a part may be overlined. Since \begin{align*}\sum_{n=0}^\infty \pd_{\overline 0,r}(n)q^n & = \frac{1}{(q^r;q^r)_\infty}\cdot \frac{(-q;q)_\infty}{(-q^r;q^r)_\infty}\\ & = \frac{(-q;q)_\infty}{(q^{2r};q^{2r})_\infty},\end{align*} it follows that $\pd_{\overline 0,r}(n)$ is also equal to the number of overpartitions of $n$ in which only parts divisible by $2r$ may be non-overlined. \end{remark}

Let $\mathcal{Q}_{\overline{0},2r}(n)$ be the set of distinct partitions of $n$ with no parts congruent to $0$ modulo ${2r}$. Thus, $\mathcal{Q}_{\overline{0},2r}(n)$ is the set of distinct  $2r$-regular partitions of $n$. Partitions that are simultaneously distinct and regular have been the subject of recent investigations (see, for example, \cite{A22,K}).
\begin{theorem}
For $n \geq 0$, we have $\pd_{\overline 0, r, e-o}(n) = Q_{\overline{0}, 2r,e-o}(n).$
\end{theorem}

\begin{proof} 
\begin{align*}
\sum_{n = 0}^\infty \pd_{\overline 0, r, e-o}(n) q^n & = \frac{1}{(-q^r; q^r)_\infty}\frac{(q; q)_\infty}{(q^r; q^r)_\infty}
= \frac{(q; q)_\infty}{(q^{2r}; q^{2r})_\infty}
=  \sum_{n=0}^\infty Q_{\overline{0}, 2r, e-o}(n)q^n.
\end{align*}
\end{proof}

\begin{proof}[Combinatorial proof] The combinatorial proof is similar to that of Theorem \ref{prd_eo}. Let $\mathcal S(n)$ be the subset of partitions $\pi\in \mathcal{PD}_{\overline 0, r}(n)$ with at least one repeated part or at least one part congruent to $0$ modulo $2r$. Then $$\mathcal{PD}_{\overline 0, r}(n)\setminus \mathcal S(n)=\mathcal Q_{\overline 0, 2r}(n).$$ If $\pi \in \mathcal S(n)$, we denote by $a=a(\pi)$ the largest repeated part of $\pi$  and by $b=b(\pi)$ the largest  part of $\pi$ congruent to $0$ modulo $2r$. If $a$ or $b$ do not exist, we set  $a=0$ or $b=0$, respectively. We define a sign reversing involution $\psi$
 on $\mathcal S(n)$ as in  Theorem \ref{prd_eo}. This completes the combinatorial proof.  
\end{proof}

\begin{corollary}
    For $n\geq 0$, $pd_{\overline 0,r}(n)$ and $Q_{\overline{0}, 2r}(n)$ have the same parity. 
\end{corollary}

As shown in  \cite[Theorem 4.1]{M17},   $\ped(n)$ can be expressed in terms of  $\textrm{pod}(n)$ by $$\ped(n)=\sum_{k=0}^\infty \textrm{pod}(n-2T_k),$$ where $T_k=k(k+1)/2$ is the $k^{th}$ triangular number.   We give an analogous result involving $\pd_{\overline 0, r}(n)$. First we introduce some notation. 
Denote by $\mathcal{PED}_{0,r}(n)$ the set of partitions of $n$ with even parts distinct and divisible by $r$ and odd parts unrestricted, and set $\textrm{ped}_{0,r}(n):=|\mathcal{PED}_{0,r}(n)|$.

\begin{theorem}\label{ped0r}
Let $n \geq 0$ and $r \geq 2$ be even. Then $$\ped_{0,r}(n)=\displaystyle \sum_{j=0}^{\infty} pd_{\overline{0}, r}(n - rT_j).$$ 
\end{theorem}

\begin{proof}[Analytic Proof]
 Using  \cite[Equation (2.2.13)]{A98} with $q$ replaced by $q^r$, we see that $\displaystyle \sum_{j = 0}^\infty q^{rT_j} = \frac{(q^{2r}; q^{2r})_\infty}{(q^r; q^{2r})_\infty}$. 
Thus, 
\begin{align*}
\displaystyle \sum_{n=0}^\infty\left(\sum_{j=0}^{\infty} pd_{\overline{0}, r}(n - rT_j)\right)q^n & = \sum_{j=0}^{\infty} q^{rT_j}\sum_{n=0}^{\infty}  pd_{\overline{0}, r}(n)q^n\\
& = \frac{(q^{2r}; q^{2r})_\infty}{(q^r; q^{2r})_\infty}\frac{1}{(q^{2r}; q^{2r})_\infty (q; q^2)_\infty}\\
& = \frac{1}{(q^{r}; q^{2r})_\infty (q; q^2)_\infty}\\ & =\frac{(-q^r;q^r)_\infty}{(q;q^2)_\infty}\\  & = \sum_{n=0}^\infty \ped_{0,r}(n)q^n.
\end{align*}
In the second to last equality, we used Euler's identity \eqref{Euler} with $q$ replaced by $q^r$. 
\end{proof}
\begin{proof}[Combinatorial Proof] Let $n\geq 0$ and $r\geq 2$ even.
In \cite{A72}, Andrews showed combinatorially that $$\ped_{e-o}(n)=\begin{cases}(-1)^{n} & \text{if }n=T_k, \ k\in \mathbb N_0 \\ 0 & \text{else}.\end{cases}$$ Denote by $\mathcal{PD}^r_{0, 2r}(n)$ the set of partitions of $n$ with all parts divisible by $r$ and parts congruent to $0$ mod $2r$ distinct, and set $\pd^r_{0,2r}(n):=|\mathcal{PD}^r_{0, 2r}(n)|$.  We modify Andrews' proof by multiplying each part of each partition by $r$ to obtain a combinatorial proof for  $$\pd^r_{0,2r,e-o}(n)=\begin{cases}(-1)^{n/r} & \text{if }n=rT_k, \ k\in \mathbb N_0 \\ 0 & \text{else}.\end{cases}$$
Notice that if $\lambda \in \mathcal{PD}^r_{0, 2r}$, then all parts of $\lambda$ are divisible by $r$ and if we write $\lambda=(\lambda^{2r\mid}, \lambda^{2r\nmid})$, then  $$\ell(\lambda)=\ell\big(\frac{1}{r} \lambda\big)=\ell\big(\frac{1}{r}\lambda^{2r\mid}\big)+\ell\big(\frac{1}{r}\lambda^{2r\nmid}\big).$$ Since  $\frac{1}{r}\lambda^{2r\mid}$ has even parts and $\frac{1}{r}\lambda^{2r\nmid}$ has odd parts, we have  \begin{align*}\ell(\lambda)& \equiv \ell\big(\frac{1}{r}\lambda^{2r\mid}\big)+\frac{1}{r}|\lambda|\pmod 2\\ & \equiv \ell(\lambda^{2r\mid})+\frac{1}{r}|\lambda|\pmod 2.\end{align*}

 In the combinatorial proof of Theorem \ref{pe}, we gave a bijection between $\mathcal{PD}_{\overline 0, r}(n)$ and $\mathcal{PE}_{0,2r}(n)$.
Then 
\begin{align*}
\displaystyle  \sum_{j=0}^{\infty}  &  pd_{\overline{0}, r}(n - rT_j) =  \\
&  |\{(\lambda, \mu)  =(\lambda^{2r\mid}, \lambda^{2r\nmid}, \mu^{2r\mid}, \mu^{2r\nmid})\vdash n  \ | \  \lambda \in \mathcal{PD}^r_{0, 2r}, \, \mu \in \mathcal{PE}_{0, 2r}, \ell(\lambda^{2r\mid}) \, \textrm{even} \}|\\
& \quad -  |\{(\lambda, \mu)=(\lambda^{2r\mid}, \lambda^{2r\nmid}, \mu^{2r\mid}, \mu^{2r\nmid}) \vdash n \ | \  \lambda \in \mathcal{PD}^r_{0, 2r}, \, \mu \in \mathcal{PE}_{0, 2r}, \ell(\lambda^{2r\mid}) \, \textrm{odd} \}|.
\end{align*}

 We will define an involution on the set $$\{(\lambda, \mu)=(\lambda^{2r\mid}, \lambda^{2r\nmid}, \mu^{2r\mid}, \mu^{2r\nmid})\vdash n \  \Big| \  \lambda \in \mathcal{PD}^r_{0, 2r}, \, \mu \in \mathcal{PE}_{0, 2r}, \, 
(\lambda^{2r\mid}, \mu^{2r\mid})\neq (\emptyset, \emptyset)\}.$$ Recall that $\lambda^{2r\mid}$ is a partition with distinct parts. 

\noindent \underline{Case 1.} If $\lambda_1^{2r\mid} \geq \mu_1^{2r\mid}$,  then

$$(\lambda^{2r\mid}, \lambda^{2r\nmid}, \mu^{2r\mid}, \mu^{2r\nmid})\mapsto (\lambda^{2r\mid}\setminus \lambda_1^{2r\mid}, \lambda^{2r\nmid}, \mu^{2r\mid}\cup (\lambda_1^{2r\mid}), \mu^{2r\nmid}).$$

\noindent \underline{Case 2.} If $\lambda_1^{2r\mid} < \mu_1^{2r\mid}$,  then

$$(\lambda^{2r\mid}, \lambda^{2r\nmid}, \mu^{2r\mid}, \mu^{2r\nmid})\mapsto (\lambda^{2r\mid}\cup \mu_1^{2r\mid}, \lambda^{2r\nmid}, \mu^{2r\mid}\setminus (\mu_1^{2r\mid}), \mu^{2r\nmid}).$$

This is an involution that changes the parity of $\ell(\lambda^{2r\mid})$. Hence, if $\mathcal O$ denotes the set of odd partitions, we have  
\begin{align*}
\displaystyle \sum_{j=0}^{\infty} pd_{\overline{0}, r}(n - rT_j)
= |\{( \lambda^{2r\nmid},  \mu^{2r\nmid} & )\vdash n  \ | \  \lambda^{2r\nmid} \in \mathcal{P}_{r, 2r}, \, \mu^{2r\nmid}\in \mathcal O  \}|.
\end{align*}

Since $\lambda^{2r\nmid}$ has all parts congruent to $r$ modulo $2r$,  the partition   $r\varphi_G(\frac{1}{r}\lambda^{2r\nmid})$ has distinct parts divisible by $r$. We have 

\begin{align*}
\displaystyle \sum_{j=0}^{\infty} pd_{\overline{0}, r}(n - rT_j)
& = |\{( r\varphi_G(\frac{1}{r}\lambda^{2r\nmid}),  \mu^{2r\nmid}  )\vdash n  \ | \  \lambda^{2r\nmid} \in \mathcal{P}_{r, 2r}, \, \mu^{2r\nmid}\in \mathcal O \}|\\ & = \ped_{0,r}(n).
\end{align*}
\end{proof}
For $n\geq 0$ and $r\geq 2$ even, let  $\mathcal{PE}_{r,2r}(n)$ be  the set of partitions of $n$ where even parts are congruent to $r$ modulo  $2r$, and set $\pe_{r,2r}(n):=|\mathcal{PE}_{r,2r}(n)|$.  If we write $\lambda\in \mathcal{PE}_{r,2r}(n)$ as $\lambda=(\lambda^{2\mid}, \lambda^{2\nmid})$,  then $\lambda^{2\mid}=r\gamma$ with $\gamma$ an odd partition. Hence $r\varphi_G(\gamma)\cup \lambda^{2\nmid}$ is a partition of $n$ in which all even parts are divisible by $r$ and distinct. This transformation is clearly invertible. Hence, $\pe_{r,2r}(n)=\textrm{ped}_{0,r}(n)$.  Then, Theorem \ref{ped0r} can be rewritten as follows.

\begin{corollary} Let $n \geq 0$ and $r \geq 2$ be even. Then $$\pe_{r,2r}(n)=\displaystyle \sum_{j=0}^{\infty} pd_{\overline{0}, r}(n - rT_j).$$ 
\end{corollary}

Considering partitions where even parts are congruent to $r$ modulo $2r$ focuses on restricting which even parts can be used instead of placing a restriction on multiplicity. This is similar to 
the focus in the definition of $4$-regular partitions. Then, in analogy to the identity $b_4(n) = \ped(n)$, in the next theorem we show that, when $r$ is even, $\pe_{r,2r}(n)$  is also equal to the number of partitions of $n$ where  the multiplicity of each even part is restricted to be exactly $r/2$. 

\begin{theorem}
For $n \geq 0$ and $r \geq 2$ even, $\pe_{r,2r}(n)$  equals the number of partitions of $n$ where even parts have multiplicity exactly $r/2.$
\end{theorem}

\begin{proof}[Analytic Proof] The generating function for the number of partitions of $n$ where even parts have multiplicity exactly $r/2$ is 

\begin{align*}
 \frac{1}{(q; q^2)_\infty}  \prod_{i = 1}^\infty & (1 + q^{2i\frac{r}{2}})\\
= &  \frac{1}{(q; q^2)_\infty} \prod_{i = 1}^\infty (1 + q^{ir})\\
= & \frac{(-q^r; q^r)_\infty}{(q; q^2)_\infty}\\
= & \frac{1}{(q; q^2)_\infty(q^r; q^{2r})_\infty},
\end{align*}

\noindent which is the generating function for the number of partitions of $n$ where even parts are congruent to $r$ modulo $2r$. 
\end{proof}

\begin{proof}[Combinatorial Proof] 
Let $\lambda$ be a partition of $n$ where even parts have multiplicity exactly $r/2$. Write $\lambda = (\lambda^{2\nmid}, \lambda^{2\mid})$.  For all $\lambda_i=2t_i \in \lambda^{2\mid}$, split $\lambda_i$ into two parts equal to $t_i$ to obtain a partition in which each part has multiplicity $r$. Merge all equal parts into a single part to obtain a distinct partition $\mu$ with parts congruent to $0$ modulo $r$. Then, $\mu=r\nu$ for a distinct partition $\nu$ and $r\varphi^{-1}_G(\nu)$ is a partition with even parts congruent to $r$ mod $2r$. Hence, $\lambda^{2\nmid}\cup r\varphi^{-1}_G(\nu)\in \mathcal{PE}_{r,2r}(n)$. Since all steps above can be reversed, the transformation is invertible.
\end{proof}

\begin{corollary}
For $n \geq 0$ and $r > 0$ even, $\displaystyle \sum_{j=0}^{\infty} pd_{\overline{0}, r}(n - rT_j)$ is equal to the number of partitions of $n$ where even parts have multiplicity exactly $r/2.$
\end{corollary}

Next, we give a final generalization of POD partitions which is very similar to the one given in Section \ref{s_3}, where we considered partitions with parts congruent to $t$ modulo $r$ distinct and other parts unrestricted; however, now we restrict the multiplicity of parts congruent to $\pm t$ modulo $r$. Because the two generalizations are so similar, we omit the proofs and instead note how to adapt the proofs of the first generalization to suit this case. 

Fix $r\geq 2$ and $t$ such that $0 < t < r / 2$.  Denote by $\mathcal{PD}_{\pm t,r}(n)$ the set of partitions of $n$ in which parts congruent to $\pm t$ modulo $r$ are distinct and all other parts are unrestricted, and  set $\pd_{\pm t,r}(n):=|\mathcal{PD}_{\pm t,r}(n)|$. Denote by  $\mathcal{P}_{\overline{\pm 2t},2r}(n)$  the set of partitions of $n$ with no parts congruent to $\pm 2t$ modulo $2r$, and set $\textrm{p}_{\overline{\pm 2t},2r}(n):= |\mathcal{P}_{\overline{\pm 2t},2r}(n) |$.

\begin{theorem}\label{thirdpod}  Let $r\geq 2$ and $0 < t < r / 2$.  If $n \geq 0$, then $pd_{\pm t,r}(n) = p_{\overline{\pm 2t},2r}(n).$
\end{theorem}

Both the combinatorial and analytic proofs of Theorem \ref{thirdpod} can be adapted from the proofs of Theorem \ref{1st_gen} by performing the same operations on parts congruent to $\pm t$ modulo $r$ that are done to parts congruent to $t$ modulo $r$.

\begin{theorem} For $n\geq 0$, $r\geq 2$, and $0 < t < r/2$, we have $pd_{\pm t,r, e-o}(n)=QE_{\pm 2t, 2r, e-o}(n)$. 
    \end{theorem}

Again, the proof follows from the proof of Theorem \ref{prd_eo} by performing the same operations on parts congruent to $\pm t$ modulo $r$ that are done to parts congruent to $t$ modulo $r$.

\begin{remark} Set $t = 1$ in Theorem \ref{thirdpod}. When $r=3,$ $pd_{\pm 1, 3}(n) = pd_{\overline 0,3}(n)$, as $m\equiv \pm 1\pmod 3$ if and only if $m\not \equiv 0 \pmod 3$. When $r=4$, $pd_{\pm 1, 4}(n)=\textrm{pod}(n)$. Moreover, when $r = 6$, the sequence $pd_{\pm 1, 6}(n)$ is listed in the OEIS \cite[A265254]{OEIS} as the number of partitions of $n$ in which even parts (if any) are not distinct. We examine these partitions further in Section \ref{s_5}. For larger $r$, the sequence does not occur in \cite{OEIS}.
\end{remark}

\medskip

\section{Generalizations of partitions with even parts not distinct}\label{s_5}

When defining  PED (and POD) partitions, we restrict even (respectively odd) parts to be distinct. If instead we reverse that idea, we can consider partitions where certain parts, based on their parity, must not be distinct. Denote by $\mathcal{PEND}(n)$ the set of partitions of $n$ in which even parts are not distinct, and set $\pend(n) := | \mathcal{PEND}(n)|.$ Similarly, denote by $\mathcal{POND}(n)$ the set of partitions of $n$ in which odd parts are not distinct, and set $\pond(n) : = |\mathcal{POND}(n)|$. We refer to these partitions as PEND (respectively POND) partitions.

As mentioned in the previous section, the  generalization of $\textrm{pod}(n)$ (for $r=6$) given by  $pd_{\pm 1, 6}(n)$, is listed in the OEIS \cite[A265254]{OEIS} as counting PEND partitions. This led us to attempt to  generalize both PEND and POND partitions by extending the multiplicity condition to other residue classes as we have for POD and PED partitions. 

Let $r\geq 2$ and $0\leq t <r$.    Denote by $\mathcal{PND}_{t, r}(n)$ the set of partitions of $n$ where parts congruent to $t$ modulo $r$ are not distinct, and set $\textrm{pnd}_{t, r} := |\mathcal{PND}_{t, r}(n)| $. 
Denote by $\mathcal{PEM}_{\pm t, r}(n)$ the set of partitions of $n$ with parts congruent to $\pm t$ modulo $r$ having even multiplicity, and set $\textrm{pem}_{\pm t, r}(n) := |\mathcal{PEM}_{\pm t, r}(n)|.$

\begin{theorem}\label{Th_prnd}
For $n \geq 0,$ $r\geq 2$, we have $\textrm{pnd}_{0, r}(n) = \textrm{pem}_{\pm r, 3r}(n).$
\end{theorem}

\begin{proof}[Analytic Proof]
The generating function for $\textrm{pnd}_{0,r}(n)$ is given by \begin{align*}\sum_{n\geq 0}\textrm{pnd}_{0, r}(n)q^n& = \frac{(q^r;q^r)_\infty}{(q;q)_\infty}\prod_{i=1}^\infty \left(\frac{1}{1-q^{ri}}-q^{ri}\right)\\ & = \prod_{j\geq 1} \frac{1 - q^{rj} + q^{2rj}}{1-q^j}\\ & = \prod_{j\geq1} \frac{(1+q^{3rj})}{(1-q^{j})(1+q^{rj})}.\end{align*}

   We have \begin{align*}\prod_{j\geq 1} & \frac{(1+q^{3rj})}{(1-q^{j})(1+q^{rj})}\\ & =\prod_{j\geq1} \frac{1}{(1-q^{j})(1+q^{3rj-r})(1+q^{3rj-2r})}\\  & =  
     \prod_{i \not\equiv 0 \text{ mod } r} \frac{1}{1-q^{i}} \prod_{j\geq1} \frac{1}{(1-q^{3rj})(1-q^{2(3rj-r)})(1-q^{2(3rj-2r)})}
   \\ & =\sum_{n\geq 0}\textrm{pem}_{\pm r, 3r}(n)q^n.
     \end{align*}
\end{proof}

\begin{proof}[Combinatorial Proof:] 
Let $\mu \in \mathcal{PND}_{0,r}(n)$. Write $\mu = (\mu^{r\mid},\mu^{r\nmid})$. Note that the parts of $\mu^{r \mid} $ all have multiplicity greater than 1.
Let $\lambda$ be the partitions obtained from $\mu$ by taking each part in $\mu^{r\mid}$ with odd multiplicity, removing three of those parts and merging them into a single part to   create a part congruent to $0$ modulo $3r$.   Then $\lambda \in \mathcal{PEM}_{\pm r, 3r}(n).$

For the inverse, start with $\lambda \in \mathcal{PEM}_{\pm r, 3r}(n).$ Write $\lambda = (\lambda^{r\mid},\lambda^{r\nmid})$.  Let $\mu $ be the partition obtained from $\lambda$
by taking each part in $\lambda^{r\mid}$ with odd multiplicity (such a part must be congruent to $0$ modulo $3r$), removing one copy and splitting  it into three equal parts.   Thus $\mu \in \mathcal{PND}_{0,r}(n)$.
\end{proof}

\begin{example}
    Let $\mu = (21^2, 20, 18^5, 11^2, 9^4, 7, 6^3, 3^6)\in \mathcal{PND}_{0,3}(253)$. Then $\mu^{3 \nmid} = (20, 11^2, 7)$ and  $\mu^{3 \mid} = (21^2, 18^5, 9^4, 6^3, 3^6)$.
For any part of $\mu^{3 \mid}$ with odd multiplicity, we remove three parts and merge them into a single part. So, $18^5$ becomes $54, 18^2$ and $6^3$ becomes $18$.
Then $\lambda = (54, 21^2, 20, 18^3, 11^2, 9^4, 7, 3^6) \in \mathcal{PEM}_{\pm 3, 9}(253).$

     Conversely, if $\lambda = (54, 21^2, 20, 18^3, 11^2, 9^4, 7, 3^6) \in \mathcal{PEM}_{\pm 3, 9}(253)$.  Then $\lambda^{3\nmid}=(20, 11^2, 7)$ and $\lambda^{3 \mid} = (54, 21^2, 18^3, 9^4, 3^6)$. For any part of $\lambda^{3\mid}$ with odd multiplicity, remove one copy and split it into three equal parts. So, $54$ becomes $18^3$ and $18$ becomes $6^3$. Then $\mu=(21^2, 20, 18^5, 11^2, 9^4, 7, 6^3, 3^6)\in \mathcal{PND}_{0,3}(253)$.
\end{example}

Recall that   $\mathcal{P}_{\overline{\pm r}, 6r}(n)$ is the set of partitions of $n$ with no parts congruent to $\pm r$ modulo $6r$ and  $\textrm{p}_{\overline{\pm r}, 6r}(n): = |\mathcal{P}_{\overline{\pm r},6r}(n)|.$ The next result follows from Theorem \ref{Th_prnd}.

\begin{corollary}\label{pnd-p}
    For $n \geq 0$ and $r\geq 2$, we have   $\textrm{pnd}_{0,r}(n)=p_{\overline{\pm r}, 6r}(n)$. 
\end{corollary}
\begin{proof} The transformation that merges each pair of equal parts congruent to $\pm r$ modulo $3r$ into a single part is a bijection from $\mathcal{PEM}_{\pm r, 3r}(n)$ to $\mathcal{P}_{\overline{\pm r}, 6r}(n)$.
    \end{proof}

\begin{remark} We may view Corollary \ref{pnd-p} as an antipode of Schur's theorem. Recall that Schur's theorem states that the number of partitions of $n$ into parts congruent to $\pm 1$ modulo $6$ equals the number of distinct partitions of $n$ into parts not congruent to $\pm 1$ modulo $3$.  If we multiply each part by $r$ we obtain an identity that states that the number of partitions of $n$ into parts congruent to $\pm r$ modulo $6r$ equals the number of distinct partitions of $n$ into parts not congruent to $\pm r$ modulo $3r$. 
    \end{remark}

    \begin{corollary} \label{cpend} For $n\geq 0$, we have $\textrm{pend}(n)= \pd_{\pm 1,6}(n)$.
    \end{corollary}
    \begin{proof} Setting $r=2$ in Corollary \ref{pnd-p}, we obtain $\textrm{pend}(n)=p_{\overline{\pm 2}, 12}(n)$. Theorem \ref{thirdpod} with $r=2$ gives $p_{\overline{\pm 2}, 12}(n)=\pd_{\pm 1,6}(n)$.
        \end{proof}

The next result is of the same flavor as Theorem  \ref{Th_prnd}. Its  analytic and combinatorial proofs are both very similar to the proofs of Theorem  \ref{Th_prnd} and we omit them.

\begin{theorem}\label{pond}
For $n \geq 0$ and $r\geq 2$, we have  $\textrm{pnd}_{r, 2r}(n) = \textrm{pem}_{\pm r, 6r}(n)$.
\end{theorem}
We single out the case $r=1$ of Theorem \ref{pond}.
\begin{corollary} \label{cpond} For $n\geq 0$, we have $\textrm{pond}(n)= \textrm{pem}_{\pm 1, 6}(n)$.
    \end{corollary}

\section{Beck-type identities}\label{s_6}

In 2017,  Beck \cite[A090867]{OEIS}  conjectured a companion identity to Euler's theorem. It states that the excess in the number of parts in all odd partitions of $n$ over the number of parts in all distinct partitions of $n$ equals the number of partitions of $n$ in which there is exactly one repeated part, and also equals the number of partitions of $n$ where there is exactly one (possibly repeated) even part. Andrews proved this conjecture \cite{A-Beck} and since then similar companion identities, referred to as  Beck-type identities, have been studied for various partition identities. 

In general, if a partition identity states that the number of partitions of $n$ with condition $X$ equals the number of partitions of $n$ with condition $Y$, then a  Beck-type companion identity states that the excess in the number of parts between the two sets of partitions of $n$ equals (a constant multiple of) the number of partitions of $n$ with a slight relaxation on $X$ (or $Y$). 

Beck-type companions for identities involving  POD  (respectively PED) partitions were proved in \cite{BW23}, so a natural question is whether or not there are Beck-type companions for identities involving the generalizations of POD and PED partitions introduced in this article.

We begin with a Beck-type companion for the identity $\pd_{0,r}(n)=b_{2r}(n)$. 
\begin{theorem} \label{B1} Let $n\geq 0$ and $r\geq 2$. The excess in the number of parts in all partitions in $\mathcal B_{2r}(n)$ over the number of parts in all partitions in $\PD_{0,r}(n)$ equals the number of partitions of $n$  with one part (possibly repeated) divisible by $2r$ and all other parts not divisible by $2r$. It is also equal to the number of partitions of $n$ with one part divisible by $r$ repeated and all other parts divisible by $r$ distinct. \end{theorem}

\begin{proof} In the combinatorial proof of Theorem \ref{1st_gen} (case $t=0$), we showed that  $$(\lambda^{r\mid}, \lambda^{r\nmid})\mapsto (r\varphi_G\big(\frac{1}{r}\lambda^{r\mid}\big), \lambda^{r\nmid})$$ is a bijection from $\mathcal B_{2r}(n)$ to $\PD_{0,r}(n)$. Then, \begin{align}\notag
    \sum_{\lambda\in \mathcal B_{2r}(n)}\ell(\lambda)& - \sum_{\lambda\in \mathcal \PD_{0,r}(n)}\ell(\lambda)\\ \label{beck-pd} & = \sum_{(\lambda^{r\mid}, \lambda^{r\nmid})\in \mathcal B_{2r}(n)}\ell\big(\frac{1}{r}\lambda^{r\mid}\big)- \sum_{(\lambda^{r\mid}, \lambda^{r\nmid})\in \mathcal \PD_{0,r}(n)}\ell\big(\frac{1}{r}\lambda^{r\mid}\big).
\end{align}
Notice that for $(\lambda^{r\mid}, \lambda^{r\nmid})\in \mathcal B_{2r}(n)$, the partition $\frac{1}{r}\lambda^{r\mid}$ has odd parts and for $(\lambda^{r\mid}, \lambda^{r\nmid})\in \PD_{0,r}(n)$, the partition $\frac{1}{r}\lambda^{r\mid}$ has distinct parts.  Moreover, in each case $\lambda^{r\nmid}$ is an $r$-regular partition. 

As stated at the beginning of this section, the original Beck identity (see \cite{A-Beck} for the analytic proof and \cite{BB,Y18} for the combinatorial proof) states that the excess in the number of parts in all odd partitions of $k$ over the number of parts in all distinct partitions of $k$ equals $|\mathcal O_{1,2}(k)|$, the number of partitions of $k$ with one even part and all other parts odd, and also $|\mathcal D_{1,2}(k)|$, the number of partitions of $k$ with one repeated part and all other parts occurring with multiplicity one.

Then,  the excess in  \eqref{beck-pd} equals $$\sum_{k=0}^{\lfloor n/r\rfloor}b_r(n-rk)|\mathcal O_{1,2}(k)|=\sum_{k=0}^{\lfloor n/r\rfloor}b_r(n-rk)|\mathcal D_{1,2}(k)|.$$ 
Since $\mathcal O_{1,2}(k)$ is in bijection with the set of partitions of $rk$ with a single part divisible by $2r$ and all other  parts congruent to $r$ modulo $2r$, and $\mathcal D_{1,2}(k)$ is in bijection with the set of partitions of $rk$ into parts divisible by $r$ with a single part repeated and all other parts distinct, this completes the proof.
\end{proof}

The Beck-type companion to Glaisher's identity \cite{Y18} states that the excess in the number of parts in all $r$-regular  partitions of $k$ over the number of parts in all  partitions of $k$ in which parts have multiplicity less than $r$ equals $(r-1)|\mathcal O_{1,r}(k)|$, $r-1$ times the number of partitions of $k$ with one  part divisible by $r$ and all other parts not divisible by $r$, and also $(r-1)|\mathcal D_{1,r}(k)|$, $r-1$ times the number of partitions of $k$ with one  part repeated at least $r$ times and all other parts occurring with multiplicity less than $r$.
Using the Beck-type companion to Glaisher's identity, we can adapt the proof of Theorem \ref{B1} to give a Beck-type companion to the identity of Theorem \ref{second pd}

\begin{theorem} Let $n\geq 0$ and $r\geq 2$. The excess in the number of parts in all partitions in $\mathcal B_{r^2}(n)$ over the number of parts in all partitions in $\mathcal{PRM}_{0,r}(n)$ equals $r-1$ times the number of partitions of $n$  with one part (possibly repeated) divisible by $r^2$ and all other parts not divisible by $r^2$. It  also equals $r-1$ times the number of partitions of $n$ with one part divisible by $r$ repeated at least $r$ times and all other parts divisible by $r$ with multiplicity less than $r$. 
    \end{theorem}

    The next theorem is a Beck-type companion to the identity of  Theorem \ref{thirdpod}.

\begin{theorem}\label{podgen} Let $n\geq 0$ and $0 < t < r/2$. The excess in the number of parts in all partitions in $\mathcal{P}_{\overline{\pm 2t}, 2r}(n)$ over the number of parts in all partitions in $\PD_{\pm t,r}(n)$ equals  the number of partitions of $n$  with one part (possibly repeated) congruent to $\pm 2t$ modulo $2r$ and all other parts not congruent to $\pm 2t$ modulo $2r$. It also equals the number of partition of $n$ with one part congruent to $\pm t$ modulo $r$ repeated and all other parts congruent to $\pm t$ modulo $r$ distinct. \end{theorem}

\begin{proof}  In \cite[Theorem 9.4]{BW23}, we gave a combinatorial proof for the Beck-type companion to the identity $$\textrm{pod}(n)=p_{\overline{2},4}(n),$$ where $p_{\overline{2},4}(n)$ is the number of partitions of $n$ with no parts congruent to $2$ modulo $4$. If we modify that proof to work with parts congruent to $\pm t$ modulo $r$ in place of odd parts and parts not congruent to $\pm 2t$ modulo $2r$ in place of parts not congruent to $2$ modulo $4$, we obtain a combinatorial proof of Theorem \ref{podgen}.
    \end{proof}

We also attempted to find a Beck-type companion to the identity of Theorem \ref{Th_prnd}. While we could establish the non-negativity of the excess in the number of parts in all partitions of $n$ in the sets involved in the identity, we were unable to find a combinatorial interpretation for the excess in the spirit of the Beck-type identities. 

\begin{proposition}
    Let $n \geq 0$. The excess in the number of parts in all partitions in $\mathcal{PND}_{0,r}(n)$ over the number of parts in all partitions in $\mathcal{PEM}_{\pm r, 3r}(n)$ is non-negative.
\end{proposition}

\begin{proof}
We use the combinatorial proof showing $\textrm{pnd}_{0, r}(n) = \textrm{pem}_{\pm r, 3r}(n)$. Recall that if we start  with $\lambda \in \mathcal{PND}_{0,r}(n)$,  for every part divisible by $r$ that has odd multiplicity, we merge three equal parts into a single part to obtain a partition in $\mathcal{PEM}_{\pm r, 3r}(n)$. So the excess in the number of parts in all partitions in $\mathcal{PND}_{0, r}(n)$ over the number of parts in all partitions in $\mathcal{PEM}_{\pm r, 3r}(n)$ is twice the number of parts divisible by $r$ with odd multiplicity in all partitions in  $\mathcal{PND}_{0,r}(n)$. 
\end{proof}

\section{Concluding remarks}\label{s_7}

In this paper, we have offered multiple generalizations for PED and POD partitions and extended several known results. We also made   connections to PEND and POND partitions and their analogous generalizations.  We hope that our work has created avenues of further study of these generalizations. 

In particular, we note two areas of interest. The first is examining the arithmetic properties of these generalizations. Much work has been done in studying arithmetic properties of PED and POD partitions. Among the may articles on the subject see, for example, \cite{A10,BM21, Chen, CG, HS, KZ, X} for congruences for PED partitions and \cite{CGM, FXY, HS10, RS, W}  for congruences for POD partitions.  Hence, this would be a natural topic of further study. The second  is generalizing  the DE-rank introduce by Andrews in \cite{A09}: if $\lambda$ is a PED partition,  the DE-rank of $\lambda$ equals $\lfloor \lambda_1/2 \rfloor - \ell(\lambda^{2\mid})$.  In \cite{A09}, several remarkable properties of functions defined in terms of the DE-rank of PED partitions are given.  Generalizing, we can define the $D_{0,r}$-rank of a partition $\lambda \in \mathcal{PD}_{0,r}$ to be $\lfloor \lambda_1/r \rfloor - \ell(\lambda^{r\mid})$. It is fairly straightforward to give an analogue of  \cite[Theorem 2.2]{A09} for  the $D_{0,r}$-rank. Analogues of the other theorems in \cite{A09} remain elusive.

\end{document}